\documentclass[12pt,a4paper,draft,reqno]{amsart}
\usepackage[english]{babel}
\usepackage{a4wide,amssymb,amsmath,xypic,ifthen}

\newcommand{\qee}{ \hfill\hspace{2pt}$\triangle$}
\newtheorem{thm}{Theorem}[section]
\newtheorem{corol}[thm]{Corollary}
\newtheorem{lemma}[thm]{Lemma}
\newtheorem{prop}[thm]{Proposition}

\theoremstyle{remark}
\newtheorem{rema}[thm]{Remark}
 \newenvironment{remark}{\begin{rema}}{\qee\end{rema}}
\newtheorem{exe}[thm]{Example}
 
\newcommand{\iso}{ \ \raise 4pt\hbox{$\sim$} \kern -10pt\hbox{$\to$}\  }
\newcommand{\cO}{{\mathcal O}}
\newcommand{\PP}{{\mathbb P}}
\newcommand{\Ext}{{\rm Ext}}

\newcommand{\cI}{{\mathcal I}}

\newcommand{\M}{{\mathfrak M}}

\newcommand{\F}{{\mathcal F}}
\newcommand{\E}{{\mathcal E}}

\newcommand{\Z}{{\mathbb Z}}

\newcommand{\C}{{\mathbb C}}
\newcommand{\FF}{{\mathbb F}}

\newcommand{\beq}{\begin{equation}}
\newcommand{\eeq}{\end{equation}}
\newcommand{\beqa}{\begin{eqnarray}}
\newcommand{\eeqa}{\end{eqnarray}}

\newcommand{\uno}{\mbox{1\kern-.59em {\rm l}}}

\newcommand{\nn}{\nonumber}
\newcommand{\be}{\begin{equation}}
\newcommand{\ee}{\end{equation}}
\newcommand{\bea}{\begin{eqnarray}}
\newcommand{\eea}{\end{eqnarray}}

\begin{document}
\begin{flushright} SISSA Preprint 56/2009/fm\\
{\tt arXiv:0909.1458v2 [math.AG]}
\end{flushright}
\title[Poincar\'e polynomial of moduli spaces of framed sheaves]{Poincar\'e polynomial of moduli spaces of framed \\[5pt] sheaves on (stacky)  Hirzebruch surfaces}
\bigskip
\date{\today}
\thanks{This research was partly supported by the INFN Research Project PI14 ``Nonperturbative dynamics of gauge theory", by   PRIN    ``Geometria delle variet\`a
algebriche", by an Institutional Partnership Grant of the Humboldt
foundation of Germany and European Commission FP7 Programme Marie
Curie Grant Agreement PIIF-GA-2008-221571
 \\[5pt] \indent E-mail: {\tt bruzzo@sissa.it}, {\tt poghos@yerphi.am}, {\tt tanzini@sissa.it} }
 \maketitle \thispagestyle{empty}
\begin{center}{\sc Ugo Bruzzo}$^\ddag$,  {\sc Rubik Poghossian}$^{\S}$ and {\sc Alessandro Tanzini}$^\ddag$
\\[10pt]  {\small
$^\ddag$ Scuola Internazionale Superiore di Studi Avanzati, \\ Via Beirut 2-4, I-34013
Trieste, Italia \\ and Istituto Nazionale di Fisica Nucleare, Sezione di Trieste
}
\\[10pt]  {\small
$^\S$  Istituto Nazionale di Fisica Nucleare, Sezione di Roma Tor Vergata,\\
Via della Ricerca Scientifica, I-00133 Roma, Italia}
\\  {\small
and  Yerevan Physics Institute, Alikhanian Br. st. 2, 0036
Yerevan, Armenia}
\end{center}

\begin{abstract} We perform a study of the moduli space of framed torsion-free sheaves on Hirzebruch surfaces by using localization techniques. 
We discuss some general properties of this moduli space by studying it in the framework
of Huybrechts-Lehn theory of framed modules. 
We classify the fixed points under a toric action on the moduli space, and use this to compute the   Poincar\'e polynomial of the latter.   This will imply that the moduli spaces we are considering are irreducible. We also consider  fractional first Chern classes, which means that we are extending our computation to a stacky deformation of a Hirzebruch surface. 
From the physical viewpoint,  our results provide the partition function of  $N=4$ Vafa-Witten theory on total spaces of the bundles $\cO_{\PP^1}(-p)$, which is relevant in black hole entropy counting problems according to a conjecture due to Ooguri, Strominger and Vafa.  \end{abstract}

\bigskip
\section{Introduction}

In this paper we study the moduli space of framed torsion-free sheaves on a Hirzebruch surface
$\FF_p$. We start by studying the geometry of this moduli space within the framework 
of Huybrechts-Lehn's theory of framed modules. We then consider a   natural toric action on 
the moduli space,  determining its fixed points and the weight decomposition of the tangent spaces at the fixed points. This will in turn allow us to compute the Morse indexes at the fixed points, and therefore the Poincar\'e polynomial of the moduli space. As a byproduct, we show that the moduli spaces we consider are irreducible.

A conjectural formula for instanton counting on the total spaces of the bundles $\cO_{\PP^1}(-p)$   (which we denote $X_p$ in this paper) was given in
\cite{fmr} and   \cite{GSST}   using string theory techniques
and   two-dimensional reduction, respectively. 
Hirzebruch surfaces    $\FF_p$ are  projective
compactifications of the   spaces $X_p$, and indeed the results of this paper provide
a proof of the conjectural formulas of \cite{fmr,GSST} for integer values of the first Chern class.
Actually our computations also make sense for fractional first Chern classes
(to be more precise, for Chern classes $c_1(E)=\frac{m}{p}C$, where $m$ is an integer,
and $C$ is the ``exceptional curve'' in $\FF_p$). This may interpreted as
a computation related to torsion-free sheaves on a stacky compactification of
$X_p$, as we shall show in section \ref{ppol}. This relates to the stacky compactifications
of ALE spaces discussed in \cite{NakaMoscow}.

The structure of this paper is as follows. In Section \ref{physics}
we provide some physical motivations. In Section \ref{modulis} we give some details
about the structure of the moduli spaces we want to study. In Section \ref{ppol} we 
compute the Poincar\'e polynomial of these moduli spaces. In Section \ref{checks} we draw
some conclusions and check our formula for the Poincar\'e polynomial in some
particular cases.

{\bf Ackowledgements.} We thank Hua-Lian Chang, Rainald Flume, Francesco Fucito, Emanuele Macr\`\i, Dimitri Markushevich, J. Francisco Morales and Claudio Rava for useful suggestions. A special thank to to Fabio Nironi for 
enlightening discussions on the stacky Hirzebruch surfaces. The second author acknowledges hospitality and support from SISSA. 

\bigskip\section{Physical motivations}\label{physics}
In this short section we want to provide some physical motivations and perspectives 
for the mathematical results we give in this paper. The reader who is not interested
in this may safely skip to the next section.

The Ooguri-Strominger-Vafa conjecture \cite{osv} relates the counting of microstates of
supersymmetric black holes to that of bound states of Dp-branes wrapping cycles of
the internal Calabi-Yau threefold in string theory compactifications.
For the D4-D2-D0 system, with the Euclidean D4 branes wrapping a four cycle $M$ of the
Calabi-Yau threefold, this counting problem  has 
a rigorous mathematical definition in terms of the generating functional
of the Euler characteristics of instanton moduli spaces on $M$
with   first and second Chern characters fixed by 
to the numer of D0-branes and D2-branes respectively.
In physical terms, this corresponds to the partition function
of the Vafa-Witten twisted $N$=4 supersymmetric theory on $M$ \cite{VW}. 

In \cite{aosv} the example of a local Calabi-Yau modeled as the total
space of a rank two bundle over a Riemann surface $\Sigma_g$
was studied in detail, providing some support for the
conjecture. 
Very little is known about the moduli space of instantons on such
spaces; indeed the authors of \cite{aosv} proceeded by conjecturing
a reduction to the base $\Sigma_g$.
An argument clarifying this reduction was provided in \cite{BonTa}
via path-integral localization. Actually, rigorous calculations of the instanton partition function
are available in a limited number of examples, namely
the case of $p=1$,  where the blowup formulas of
\cite{VW,Yoshi}  apply, and $p=2$.
Indeed,   $X_2$ is the $A_1$ ALE space
\cite{KN, naka-ale}
 (for details on  instanton counting
on $A_{p-1}$ ALE spaces the reader may refer to \cite{KN,Fucito:2004ry,fmr}).

One relevant feature of instantons on the spaces $X_p$   is the appearance of solutions with fractional 
first Chern class. These are related to instantons that asymptote to flat connections with nontrivial holonomy
at the boundary of $X_p$, which is topologically $S^3/\Gamma$ . The flat connections are classified by irreducible representations
of the finite group $\Gamma$. In this paper we propose an algebro-geometric counterpart
of these solutions by considering 
the global quotient $\mathbb{P}^2/\Gamma$ and performing a minimal resolution
of the singularity at the origin. The resulting variety is a toric Deligne-Mumford stack $\mathcal X_p$ whose coarse space
is the Hirzebruch surface $\FF_p$; it corresponds to
a ``stacky" compactification of $X_p$ obtained by adding a divisor $\tilde C_\infty\simeq \mathbb{P}^1/\Gamma$. 

Torsion free sheaves framed on $\tilde C_\infty$ provide solutions with fractional first Chern class.
We will compute the Poincar\'e polynomial of the moduli space of these framed sheaves via localization and show that
the resulting generating function of Euler characteristics is in agreement with the conjectural formulae
of \cite{fmr} and \cite{GSST}. This has a nice expression in terms of modular forms.
The details of this construction are presented in Sect.4.

\bigskip\section{Moduli of framed sheaves on Hirzebruch surfaces}  \label{modulis}
In this section we briefly study the structure of the moduli space of framed sheaves on Hirzebruch surfaces. 

\subsection{Hirzebruch surfaces}
We denote by $\FF_p $ the $p$-th Hirzebruch surface
$\FF_p=\PP(\cO_{\PP^1} \oplus\cO_{\PP^1}(-p))$, which is the projective closure
of the total space $X_p$ of the line bundle $\cO_{\PP^1}(-p)$ on $\PP^1$
(a useful reference about Hirzebruch surfaces is \cite{BHPV}).
This may be explicitly described as the divisor in $ \PP^2\times\PP^1$
$$
\FF_p =  \{ ([z_0 : z_1 : z_2], [z : w] \in \PP^2\times\PP^1\mid
   z_1 w^p = z_2 z^p \},
$$
Denoting by $f: \FF_p \to \PP^2$   the projection onto $\PP^2$,
we let $C_\infty=f^{-1}(l_\infty)$, where $l_\infty$ is the  ``line at infinity" 
$z_0=0$. The Picard group of $\FF_p$ is generated
by $C_\infty$ and the fibre $F$ of the projection $\FF_p\to\PP^1$.
One has
$$C_\infty^2=p,\qquad C_\infty\cdot F=1,\qquad F^2=0\,.$$
The canonical divisor $K_{p}$ may be expressed as
$$
K_{p} = - 2 C_\infty + (p-2) F.
$$

We shall consider the moduli space $\M^p(r,k,n)$ parametrizing isomorphism classes
of pairs $(\E,\phi)$, where 
\begin{itemize}\item $\E$ is a torsion-free coherent sheaf on $\FF_p$, whose
topological invariants are the rank $r$, the first Chern class $c_1(\E)=kC$, and the
discriminant
$$ \Delta(\E)=c_2(\E)-\frac{r-1}{2r}c_1^2(\E)= n;$$
\item $\phi$ is a framing on $C_\infty$,  i.e., an isomorphism of the restriction of $\E$ to
$C_\infty$ with the trivial rank $r$ sheaf on $C_\infty$:
$$\phi\colon \E_{\vert C_\infty} \iso \cO_{C_\infty}^{\oplus r}.$$
\end{itemize}
In constructing the moduli space $\M^p(r,k,n)$ one only considers isomorphisms which preserve
the framing, i.e., $(\E,\phi)\simeq(\E',\phi')$ if there is an isomorphism $\psi\colon\E\to\E'$ such that
$\phi'\circ\psi=\phi$. While a point in $\M^p(r,k,n)$  should always be denoted as the class
of a pair $(\E,\phi)$, we shall be occasionally a little sloppy in our notation,  omitting the framing $\phi$.

For $r=1$, the value of $k$ is irrelevant, and every moduli space  $\M^p(1,k,n)$
turns out to be isomorphic to the Hilbert scheme $X_p^{[n]}$ parametrizing
length $n$ 0-dimensional subschemes of $X_p$. The structure of the moduli space
for $r>1$ can be studied using an ADHM description \cite{Rava}, which shows that the moduli space is a smooth algebraic variety of dimension $2rn$. However we are not going
to give this description here.

\subsection{Structure of the moduli space} A general result about the structure
of moduli spaces of framed sheaves on a projective surface $X$ was given by Nevins \cite{Nev}, who
showed that, under some mild rigidity condition, there is a fine moduli scheme.
However in the specific case at hand much more can be said.
The structure of the moduli space may be studied by showing that,  under some conditions,
and after a suitable choice a polarization in $X$, and a certain stability parameter, framed sheaves yield
stable pairs according to the theory developed by Huybrechts and Lehn \cite{HL1,HL2}. 
The moduli space turns out to be a quasi-projective separated scheme of finite type \cite{BM}.
 The obstruction to smoothness at a point $[(\E,\phi)]$ of the moduli space
lies in the kernel of the trace map $$\operatorname{Ext}^2(\E,\E(-D))\to H^2(X,\cO(-D))$$
where $D$ is the framing divisor.
In our case $\operatorname{Ext}^2(\E,\E(-C_\infty))=0$, hence the moduli space is smooth
(and then is an algebraic variety).

 The result one gets in our case is the following.
 
\begin{prop} \label{moduli} The moduli space $\M^p(r,k,n)$, when nonempty, is a smooth   quasi-projective variety
of dimension $2rn$. Its tangent space at a point $[\E]$ is isomorphic to the vector space
$\Ext^1(\E,\E(-C_\infty))$. 
\end{prop}

 \subsection{Toric action} \label{torusaction} 
 The two-dimensional algebraic torus $\C^\ast\times\C^\ast$ acts on $\FF_p$  according to
$$
([z_0 : z_1 : z_2], [z : w]) \stackrel{G_{t_1,t_2}}{\relbar\joinrel\relbar\joinrel\relbar\joinrel\longrightarrow}
([z_0 : t_1^p z_1 : t_2^p z_2], [t_1 z : t_2 w])
$$
The divisor $C=f^{-1}([1:0:0])$ and $C_\infty$ are fixed under this action. 
Note that one has
$C=C_\infty-pF$ as divisors modulo linear equivalence. 
Moreover, this action has four fixed points, i.e., $p_1=([1:0:0],[0:1]$ and $p_2=([1:0:0],[1:0]$
 lying on the exceptional line $C$, and two points lying on the line at infinity $C_\infty$. 
 The invariance of $C_\infty$ implies that 
the pullback  $G_{t_1,t_2}^\ast$ defines an action on  $\M^p(r,k,n)$. Moreover we have an action of the diagonal maximal torus of
 $Gl(r,\C)$ on the framing. Altogether we have an action of the torus $T=(\C^\ast)^{r+2}$ on 
 $\M^p(r,k,n)$ given by 
\begin{equation}\label{action}
    (t_1,t_2,e_1,\dots,e_r)\cdot (\E,\phi)
    = \left((G_{t_1,t_2}^{-1})^\ast \E, \phi'\right),
\end{equation}
where $\phi'$ is defined as the composition 
$$   (G_{t_1,t_2}^{-1})^* \E_{\vert C_\infty} 
   \xrightarrow{(G_{t_1,t_2}^{-1})^*\phi}
   (G_{t_1,t_2}^{-1})^* \cO_{C_\infty}^{\oplus r}
   \longrightarrow \cO_{C_\infty}^{\oplus r}
   \xrightarrow{e_1,\dots, e_r} \cO_{C_\infty}^{\oplus r}.
$$
 
 We study now the fixed point sets for the action of $T$ on $\M^p(r,k,n)$. This is basically the same statement as in \cite{NY-I} (see also \cite{NY-L} and \cite{GaLiu}),  but for the sake of completeness we give here a sketch of the proof .
  
 \begin{prop} The fixed points of the action \eqref{action} of $T$ on $\M^p(r,k,n)$ are
 sheaves of the type
\begin{equation}\label{fixpo} \E = \bigoplus_{\alpha =1}^r \cI_{\alpha } (k_\alpha C)\end{equation}
 where $\cI_{\alpha }$ is the ideal sheaf of a 0-cycle $Z_\alpha $ supported on $\{p_1\}\cup\{p_2\}$
 and $k_1,\dots,k_r$ are integers which sum up to $k$. The $\alpha$-th factor in this decomposition corresponds, via the framing, to the $\alpha$-th factor in $\cO_{C_\infty}^{\oplus r}$. Moreover, 
\begin{equation}\label{nfixpo} \  n = \ell +\frac{p}{2r}\left(r\sum_{\alpha =1}^rk_\alpha ^2-k^2\right)
= \ell  + \frac{p}{2r} \sum_{\alpha <\beta} (k_\alpha  -k_\beta)^2\,
\end{equation}
 where $\ell$ is the length of the singularity set of $\E$. 
\label{fp} 
\end{prop}
 \begin{proof} We first check that if $\E$ is fixed under the $T$-action, then
 $\E=\oplus_{\alpha =1}^r\E_\alpha $, where each $\E_\alpha $ is a $T$-invariant rank-one torsion-free sheaf
 on $\FF_p$. Let $\mathcal K$ be the sheaf of rational functions on $\FF_p$. Then $\E'=\E\otimes\mathcal K$ is free as a $\mathcal K$-module. Choose a trivialization $\E'\simeq \oplus_{\alpha =1}^r\E_\alpha '$ which, when restricted to $C_\infty$, provides an eigenspace decomposition for the action of $T$.
 Then set $\E_\alpha =\E_\alpha '\cap\E$, i.e., pick up the holomorphic sections of $\E'_\alpha $. This provides the desired decomposition. 
 
By taking the double dual of $\E$, and using $T$-invariance, we obtain  $ \E^{\ast\ast} \simeq \bigoplus_{\alpha =1}^r \cO_{\FF_p} (k_\alpha C)$, whence $\E$ has the form \eqref{fixpo}. Since the 0-cycles $Z_\alpha$ have to be $T$-invariant and should not be supported on $C_\infty$, they must be supported as claimed. The numerical equality \eqref{nfixpo} follows from a straight calculation.
 \end{proof}
 
 The exact identification of the fixed points is obtained by using some Young tableaux combinatorics. The precursor of this technique is Nakajima's seminal book \cite{NakaBook} where the case of the Hilbert scheme of points of $\C^2$ is treated (from a gauge-theoretic viewpoint, this is the ``rank 1 case'' for framed instantons on $S^4$). This was generalized to higher rank in \cite{Nek,Flupo,BFMT}. 
 
 A word on notation: if $Y$ is a  Young tableau, $\vert Y \vert $ will denote the   number of boxes in it.  In the case at hand, it turns out that to each fixed point one should attach
 an $r$-ple $\{Y^{(i)}_\alpha\}$ of pairs of Young tableaux (so $i=1,2$ and $\alpha=1,\dots,r$).
 Write $Z_\alpha=Z_\alpha^{(1)}\cup Z_\alpha^{(2)}$, where $Z_\alpha^{(i)}$ is supported at $p_i$.
 The Young tableau  $\{Y^{(i)}_\alpha\}$, for $i$ and $\alpha$ fixed, is attached to the ideal sheaf
 $\cI_{Z_\alpha^{(i)}}$ as usual: choose local affine coordinates $(x,y)$ around $p_i$
 and make a correspondence between the boxes of $\{Y^{(i)}_\alpha\}$ and monomials in $x$, $y$
 in the usual way (cf.~\cite{NakaBook}); then $\cI_{Z_\alpha^{(i)}}$ is generated by the monomials lying outside the tableau. Now the identity \eqref{nfixpo} may be written as 
\begin{equation}\label{countboxes}
n =  \sum_\alpha  \left( \vert
Y_\alpha ^1 \vert +\vert Y_\alpha ^2 \vert \right) +
\frac{p}{2r} \sum_{\alpha <\beta} (k_\alpha  -k_\beta)^2\,.\end{equation}
Looking for all collections of Young tableaux and strings of integers $k_1,\dots,k_r$ satisfying this condition together with $\sum_{\alpha=1}^rk_\alpha=k$, one enumerates all
the fixed points.

\subsection{Nonemptiness of the moduli space} The moduli spaces 
$\M^p(r,k,n)$ may be given an ADHM description, as shown in \cite{Rava}. 
The analysis performed in \cite{Rava} implies the following results.
Note  that the value of $k$ may be normalized in the range $ 0 \le k < r-1$ upon twisting
by $\cO_{\FF_p}(C)$, and in the sequel we shall assume that this has been done.
\begin{prop}  \begin{enumerate} \item
The moduli space $\M^p(r,k,n)$ is nonempty if and only if the number 
$n-\frac{r-1}{2r}pk^2$ is an integer, and the bound
\begin{equation}\label{bound} n \ge N=\frac{pk}{2r}(r-k)\end{equation}  holds. 
\item All sheaves in $\M^p(r,k,N)$ are locally free.
\item The moduli space $\M^1(r,k,N)$ is isomorphic to the Grassmannian
variety $G_k(r)$ of $k$-planes in $\C^r$. For $p>1$, the space $\M^p(r,k,N)$
is isomorphic to a rank $(p-1)$ vector bundle on $G_k(r)$. In particular,
for all $p\ge 1$ the moduli space  $\M^p(r,k,N)$ has the homotopy type of
a compact space.
\end{enumerate}\label{boundprop}
\end{prop}

\bigskip\section{Poincar\'e polynomial} \label{ppol}
In this section we determine the weight decomposition of the toric action on the tangent space to the moduli space at the fixed points,
and use this information to compute the Poincar\'e polynomial of the moduli spaces $\M^p(r,k,n)$.
 
 \subsection{Defining the Poincar\'e polynomial} Let $X$ be a space whose cohomology with rational coefficients is finite-dimensional. One defines the Poincar\'e polynomial of $X$ as
 $$P_t(X) = \sum_{n\ge 0} (\dim H^n(X,\mathbb Q))\,t^n\,.$$
\subsection{Sheaves on stacky Hirzebruch surfaces}
Actually our computations also make sense for $c_1(\E)=kC$ with $k=m/p$ for integer $m$, and $p\ge 2$.
This can be justified by considering a ``stacky compactification''  of    $X_p$;
instead of adding the divisor $C_\infty$, we add 
$\tilde C_\infty\simeq C_\infty/\mathbb Z_p$. One obtains a Deligne-Mumford stack $\mathcal X_p$,   whose so-called coarse space   may be identified with the Hirzebruch surface $\FF_p$, 
and one has a 
a morphism $\pi\colon\mathcal X_p\to \FF_p$.
Since
the push-forward functor $\pi_* \colon \text{Coh}(\mathcal X_p) \to \text{Coh}(\mathcal \FF_p) $ is exact
\cite{av},
one has  an isomorphism 
\be
H^i(\mathcal X_p, \F) \simeq  H^i(\FF_p, \pi_\ast \F) 
\label{iso}
\ee
for any coherent sheaf $\F\in \text{Coh}(\mathcal X_p)$ and all $i$.  Owing to this basic fact, we can reduce all our computations to the coarse
space $\FF_p$, by appropriately taking into account the push-forward of the relevant sheaves. 
The theory developed in  \cite{famani} allows one to prove the following Lemma.
\begin{lemma}
The Picard lattice of $\mathcal X_p$ is freely generated over $\Z$ by the divisor
$\tilde C_\infty$ and the inverse image $\tilde F=\pi^{-1}(F)$. Moreover one has 
\begin{equation}
\label{pushf}\begin{array}{rcl}
\pi(p\,\tilde C_\infty) &=& C_\infty \\[3pt]
\pi(\tilde F) &=& F \, .
\end{array}\end{equation}
\end{lemma}

Let $\widetilde\M^p(r,k,n)$ be the moduli space of torsion-free rank $r$ sheaves $\E$ on $\mathcal X_p$, with 
$c_1(\tilde \E)=k C$ and
discriminant $n$, that are framed on $\tilde C_\infty$ to the sheaf $\mathcal O_{\tilde C_\infty}^{\oplus r}$. 
Here $k=m/p$ for some integer $m$. The fixed points
under the torus action are as in Proposition \ref{fp}, except that in this case the
$k_\alpha$'s are $k_\alpha=m_\alpha/p$, $m_\alpha\in \mathbb{Z}$. 
 
 \begin{remark} We do not have for the moduli spaces $\widetilde\M^p(r,k,n)$ a fully developed theory as we described in the previous section for the spaces $\M^p(r,k,n)$. On the other hand,
 when $k$ takes integer values, in the following treatment $\widetilde\M^p(r,k,n)$ 
 may be replaced by $\M^p(r,k,n)$, so that all results are rigourous. Therefore, the results for non-integer values of $k$ are to be regarded as heuristic, in the sense
 that they need a more rigourous proof, involving the construction of moduli spaces
 of framed sheaves on the stacks $\mathcal X_p$.
 \end{remark}

\subsection{Weight decomposition of the tangent spaces} 
We compute here the weights of the action of the torus $T$ on the irreducible subspaces of the tangent spaces to $\widetilde\M^p(r,k,n)$ 
at the fixed points $(\E,\phi)$ of the action. According to the decomposition \eqref{fixpo},  the tangent space $T_{(E,\phi)} \widetilde\M^p(r,k,n) \simeq \Ext^1(\E,
\E(-\tilde C_\infty))$ splits as 
$$
   \Ext^1(\E, \E(-C_\infty))
   = \bigoplus_{\alpha ,\beta}
        \Ext^1(\cI_\alpha (k_\alpha  C), \cI_\beta(k_\beta C - \tilde C_\infty)).
$$
The factor $\Ext^1(\cI_\alpha (k_\alpha  C), \cI_\beta(k_\beta C - \tilde C_\infty))$
has weight $e_\beta e_\alpha ^{-1}$ under the maximal torus of $Gl(r,\C)$. So we need only
 to describe the weight decomposition with respect to the remaining action of 
 $T^2=\C^\ast\times \C^\ast$. 

From the exact sequence
$
0\to \cI_\alpha  \to \cO \to \cO_{Z_\alpha }\to 0
$
we have in K-theoretic terms
\begin{multline}\label{exts}
\Ext^\bullet(\cI_\alpha (k_\alpha  C),   \cI_\beta(k_\beta C - \tilde C_\infty))  =
\Ext^\bullet(\cO(k_\alpha  C), \cO(k_\beta C - \tilde C_\infty))  \\  
     -   \Ext^\bullet(\cO(k_\alpha  C), \cO_{Z_\beta}(k_\beta C - \tilde C_\infty))     -  
  \Ext^\bullet(\cO_{Z_\alpha }(k_\alpha  C), \cO(k_\beta C - \tilde C_\infty))   \\   
   +   \Ext^\bullet(\cO_{Z_\alpha } (k_\alpha  C),
           \cO_{Z_\beta}(k_\beta C - \tilde C_\infty)).  
\end{multline}
We should note that 
\begin{eqnarray*}
\Ext^0 (\cO(k_\alpha  C), \cO(k_\beta C - \tilde C_\infty)) &\simeq& 
H^0(\mathcal X_p, \cO(- n_{\alpha\beta} C - \tilde C_\infty)) \simeq \nonumber\\
&\simeq& H^0\left(\FF_p, \cO\left(- ([n_{\alpha\beta}]+1) C_\infty + p n_{\alpha\beta} F\right)\right) =
0 ,
\end{eqnarray*}
where $n_{\alpha\beta}=k_\alpha  - k_\beta$ and $[n_{\alpha\beta}]$ is its integer part.
In order to show the above vanishing we used the relation $C=p(\tilde C_\infty - F)$, the isomorphism
(\ref{iso}) and the Lemma \ref{pushf}. More precisely, for our purposes we only need to compute the  
push-forward of line bundles; this is done along the lines of \cite{borne}.
Analogously
$$
\Ext^2 (\cO(k_\alpha  C), \cO(k_\beta C - \tilde C_\infty)) \simeq H^0(\FF_p,\cO (( [n_{\alpha\beta}]+1)C_\infty - p n_{\alpha\beta} F)) = 0\,.
$$
since the divisors $C_\infty$ and $F$ intersect.
We are thus left with the computation of the $\Ext^1$ groups
$$
\Ext^1(\cO(k_\alpha  C),\cO(k_\beta C - \tilde C_\infty)) =
H^1(\FF_p, \cO(-n_{\alpha\beta} C - \tilde C_\infty))
$$
We distinguish three cases according to the values of $[n_{\alpha\beta}]$. In the first case
 ($[n_{\alpha\beta}]=0$) one again easily sees that  
$H^1(\mathcal X_p, \cO( - n_{\alpha\beta} C - \tilde C_\infty))=H^1(\FF_p, \cO(p \{n_{\alpha\beta}\}F - C_\infty))=
0$, where $\{n_{\alpha\beta}\}$ is the fractional part,  $\{n_{\alpha\beta}\} = n_{\alpha\beta}
- [n_{\alpha\beta}]$.
  
In the second case ($[n_{\alpha\beta}]>0$) we get 
\be
H^1(\mathcal X_p, \cO(- n_{\alpha\beta}C - \tilde C_\infty)) \simeq
\oplus_{d=0}^{[n_{\alpha\beta}]-1} H^0(\PP^1, \cO_{\PP^1}(pd + p\{n_{\alpha\beta}\}))
\label{ngt1}
\ee 
We prove this result by induction on $[n_{\alpha\beta}]>0$ using the  exact sequence
\begin{equation}
0\to \cO(-(n_{\alpha\beta}+1) C - \tilde C_\infty) \to \cO (- n_{\alpha\beta}C - \tilde C_\infty) \to
\cO_{C}(-n_{\alpha\beta}C ) \to 0
\label{seqn}
\end{equation}
For $n_{\alpha\beta}=1$ one readily obtains \eqref{ngt1}. On the other hand from
\eqref{seqn} we get
\begin{multline*} 0 \to H^0 (\PP^1, \cO(n_{\alpha\beta}p)) \to H^1(\mathcal X_p, \cO(-(n_{\alpha\beta}+1)C - \tilde C_\infty)) \\
\to H^1(\mathcal X_p, \cO(-n_{\alpha\beta}C - \tilde C_\infty)) \nn  
  \to H^1(\PP^1, \cO(n_{\alpha\beta}p))   = 0
\end{multline*} 
and by the inductive hypothesis  
\begin{multline*}
 0 \to H^0 (\PP^1, \cO(n_{\alpha\beta}p)) \to H^1(\mathcal X_p, \cO(-(n_{\alpha\beta}+1)C - \tilde C_\infty)) \\
\to \oplus_{d=0}^{[n_{\alpha\beta}]-1} H^0(\PP^1, \cO_{\PP^1}(pd+ p\{n_{\alpha\beta}\})) \to 0
\end{multline*} 
so that \eqref{ngt1} is proved. 
Since $H^0(\PP^1, \cO_{\PP^1}(pd+ p\{n_{\alpha\beta}\}))$ is the space of homogeneous
polynomials  of degree $pd+p\{n_{\alpha\beta}\}$ in two variables, it equals  $\sum_{i=0}^{pd+p\{n_{\alpha\beta}\}}
t_1^{-i} t_2^{-pd+p\{n_{\alpha\beta}\}+i}$ in the representation ring of $T^2$.
Finally, we get the weights of this summand of the tangent space as
\begin{equation}
L_{\alpha, \beta} (t_1,t_2) 
=  e_\beta\, e_\alpha ^{-1}
\sum_{{i,j\ge 0, i+j-pn_{\alpha\beta} \equiv 0 \ {\rm mod} p, \ i+j \le p(n_{\alpha\beta}-1)}}
      t_1^{-i} t_2^{-j} 
\label{Lmag0}
\end{equation}

For the third case ($n_{\alpha\beta}<0$) we have analogously
\begin{equation}
L_{\alpha, \beta} (t_1,t_2) 
 =   e_\beta\,e^{-1}_\alpha
\sum_{{i,j\ge 0,\ i+j+2+pn_{\alpha\beta}\equiv 0 \ {\rm mod} p, \ i+j \le - pn _{\alpha\beta}- 2}}
      t_1^{i+1} t_2^{j+1} 
\label{Lmin0}
\end{equation}

These terms contribute to the weight decomposition of $T_{(E,\phi)} \widetilde\M_p(r,k,n)$
at the fixed points together with  the remaining terms in \eqref{exts}, which are
essentialy the same as in $\PP^2$ case modulo some rescalings of the arguments
(in a sense, we   look at $X_p$ as the resolution of singularities of $\C^2/\Z_p$,
and rescale the arguments to achieve $\Z_p$-equivariance).
In this way we get
\begin{multline} 
   T_{(E,\phi)} \widetilde\M^p(r,k,n)
   = \\ \sum_{\alpha ,\beta=1}^r \left(L_{\alpha ,\beta}(t_1,t_2)
     + t_1^{p(k_\beta - k_\alpha )} N_{\alpha ,\beta}^{\vec{Y}_1}(t_1^p,t_2/t_1)
     + t_2^{p(k_\beta - k_\alpha )}
     N_{\alpha ,\beta}^{\vec{Y}_2}(t_1/t_2,t_2^p)\right),\label{tangent_space}
\end{multline} 
where $L_{\alpha ,\beta}(t_1,t_2)$ is given by (\ref{Lmag0})
and (\ref{Lmin0}), while
 \be
N_{\alpha ,\beta}^{\vec{Y}}(t_1,t_2)=e_\beta e_\alpha ^{-1}\times
\left\{\sum_{s \in Y_\alpha }
\left(t_1^{-l_{Y_\beta}(s)}t_2^{1+a_{Y_\alpha }(s)}\right)+\sum_{s
\in Y_\beta}
\left(t_1^{1+l_{Y_\alpha }(s)}t_2^{-a_{Y_\beta}(s)}\right)\right\}\,,
\label{R4_tangent_space} \ee
a well known expression for the
$\C^2$ case (first introduced in \cite{Flupo}). 
Here $\vec Y$ denotes an $r$-ple  of Young tableaux, while for a given box $s$ in the tableau
$Y_\alpha$, the symbols $a_{Y_\alpha}$ and $l_{Y_\alpha}$ denote the ``arm" and ``leg" of $s$ respectively, that is, the number of boxes above and on the right to $s$.

\subsection{Computing the Poincar\'e polynomial}
We shall closely follow the method presented in \cite{NY-L},
Section  3.4. We choose a one-parameter subgroup $\lambda (t)
$ of the $r+2$ dimensional torus $T$
$$ \lambda (t)=(t^{m_1},
     t^{m_2},t^{n_1},\dots,t^{n_r}),
$$ by specifying generic weights such that
 $$ m_1=m_2\gg n_1>n_2\cdots>n_r>0 \,.$$
We have now a different fixed locus; as far as the action on the surface
$\FF_p$ is concerned, the entire ``exceptional line'' $C$ is pointwise  fixed
under this action. Accordingly, the fixed points in $\M^(r,k,n)$ 
still have the form \eqref{fixpo},  with $Z_\alpha$ fixed under the action of the diagonal subgroup $\Delta$ of $\C^\ast\times\C^\ast$, which means that the points in the supports of the 0-cycles
$Z_\alpha$ are arbitrary points in $C$. 
Each component of the fixed point set is parametrized by an $r$-ple of pairs
$((k_1,Y_1),\dots,(k_r,Y_r))$ with
$$n=\sum_{\alpha=1}^r\vert Y_\alpha\vert+\frac{p}{2r}\sum_{\alpha<\beta}(k_\alpha-k_\beta)^2
\quad\text{and}\quad \sum_{\alpha=1}^r k_\alpha=k\,.$$

\begin{remark} This action allows one to prove that the space $\M^p(r,k,n)$
has the homotopy type of a compact space, thus generalizing point (iii) of Proposition \ref{boundprop}.
This is needed to fully justify our computation of the Poincar\'e polynomial of $\M^p(r,k,n)$
which uses, albeit indirectly, Morse theory.
We may define a set 
$$\M_0^p(r,k,n) =
 {\lower8pt\hbox{\Huge$\amalg$} \atop {\mbox{\tiny$0 \le \ell \le n - N$}}}
\M^p_{\mbox{\small lf}}(r,k,n-\ell)
\times \operatorname{Sym}^\ell X_p$$
(where $\M^p_{\mbox{\footnotesize lf}}(r,k,n)$ is the open subscheme of $\M^p(r,k,n)$
given by the locally free points)
and a map $\gamma\colon \M^p(r,k,n) \to \M_0^p(r,k,n) $
by letting
$$ \gamma((\E,\phi) )= \left((\E^{\ast\ast},\phi),\,\operatorname{supp}(\E^{\ast\ast}/\E)\right.)\,.$$
The space $\M_0^p(r,k,n)$ may be given a schematic structure by looking at it as
a Uhlenbeck-Donaldson partial compactification of $ \M^p_{\mbox{\footnotesize lf}}(r,k,n)$, and the
morphism $\gamma$ is then projective (hence proper) \cite{BMT,HL-book}. There is a natural action
of $\C^\ast$ on $ \M_0^p(r,k,n) $, which is compatible with the action
on $ \M^p(r,k,n)$. The ``deepest'' stratum in $\M_0^p(r,k,n)$ is isomorphic to
$\M^p(r,k,N)\times\operatorname{Sym}^{n-N}X_p$, and its intersection $Z$ with the
fixed locus of the $\C^\ast$ action is isomorphic to a product of symmetric products of copies of  the exceptional
curve $C$ (hence, is compact).
An analysis analogous to the one developed in \cite{NakaBook,NY-I} shows that 
$\gamma^{-1}(Z)$ has the same homotopy type of $\M^p(r,k,n)$.
\end{remark}

To link with the previous construction, we should notice that each fixed point  $\bar\E$  of the
$\Delta\times T^r$ action corresponding to an $r$-ple $((k_1,Y_1),\dots,(k_r,Y_r))$ 
determines a fixed point $((k_1,\emptyset,Y_1),\dots,(k_r,\emptyset,Y_r))$ of the
$T$-action which lies in the same component of the fixed locus as $\bar\E$. Since the $\Delta\times T^r$-module structure of the tangent space $T_{(\E,\phi)}\widetilde{\M}^p(r,k,n)$ at a fixed point of $\Delta\times T^r$
does not  change if we move the fixed point in its component, we may choose   $((k_1,\emptyset,Y_1),\dots,(k_r,\emptyset,Y_r))$. In this case
(\ref{tangent_space})
reduces to 
\be
   T_{(E,\phi)} \widetilde{\M}_p(r,k,n)
   = \sum_{\alpha ,\beta=1}^r \left(L_{\alpha ,\beta}(t_1,t_1)+
   t_1^{p(k_\beta - k_\alpha )} N_{\alpha ,\beta}^{\vec{Y}}(1,t_1^p)\right), \label{reduced_tangent_space}
\ee By (\ref{R4_tangent_space}) we have 
$$
N_{\alpha ,\beta}^{\vec{Y}}(1,t_1^p)=e_\beta e_\alpha ^{-1}\times
\left(\sum_{s \in Y_\alpha } t_1^{p(1+a_{Y_\alpha }(s))}+\sum_{s \in
Y_\beta} t_1^{-p\,a_{Y_\beta}(s)}\right). 
$$
 Our task now is to
compute the index of the critical points, that is,  the number of terms
in (\ref{reduced_tangent_space}) for which one of the following possibilities
holds: 
\begin{enumerate}
\item the weight of $t_1$ is negative,
\item the weight of $t_1$ is zero and the weight of $e_1$ is negative, 
\item the weights of $t_1$, $e_1$ are zero and the weight of
$e_2$ is negative, 
\item the weights of $t_1$, $e_1$, $e_2$ are zero and weight of $e_3$ is
negative,  
\item[ ] ... ...
\item[($r+1$)] the weights of $t_1$, $e_1$,...,$e_{r-1}$ are zero and the weight
of $e_r$ is negative.
\end{enumerate} 
The contribution of the    diagonal ($\alpha =\beta$)
terms of (\ref{reduced_tangent_space})
turns out to be
$$
\sum_{\alpha =1}^r (\mid Y_\alpha \mid -l(Y_\alpha )).
$$
The nondiagonal ($\alpha \neq\beta$)  terms  can be
rewritten as 
\begin{multline}
\label{nondiagonal_tangent_space}
   \sum_{\alpha <\beta}
   \left(\phantom{\Biggl(}L_{\alpha ,\beta}(t_1,t_1)+L_{\beta,\alpha }(t_1,t_1) \right.
   \\ 
   +   \sum_{s \in Y_\alpha }\left(\frac{e_\beta}{e_\alpha }\, t_1^{p(1+a_{Y_\alpha }(s)+k_\beta -
k_\alpha )}+\frac{e_\alpha }{e_\beta}\,t_1^{p\,(-a_{Y_\alpha }(s)-k_\beta
+ k_\alpha )}\right)  
\\  \left.
+\sum_{s \in
Y_\beta}\left(\frac{e_\beta}{e_\alpha }
\,t_1^{p(-a_{Y_\beta}(s)+k_\beta -
k_\alpha )}+\frac{e_\alpha }{e_\beta}\,t_1^{p\,(1+a_{Y_\beta}(s)+k_\alpha 
- k_\beta)}\right)\right),
\end{multline}

We compute now the $L$ terms in these expressions.  Due to the
formulas (\ref{Lmag0}), (\ref{Lmin0}),  if $k_\alpha  -k_\beta \ge 0$
only the term $L_{\alpha ,\beta}(t_1,t_1)$ contributes to the
index. This contribution is easy to count:
$$
 \sum_{{i,j\ge 0,
i+j-pn_{\alpha\beta}\equiv 0 \ {\rm mod} p, \ i+j \le
p(n_{\alpha\beta}-1)}}1
=\frac{1}{2}[n_{\alpha\beta}]\left(p [n_{\alpha\beta}] +2 - p\right) + p[n_{\alpha\beta}]\{n_{\alpha\beta}\}.
$$
If, instead, $k_\alpha  -k_\beta < 0$ there is a contribution from
the $L_{\beta ,\alpha }(t_1,t_1)$ term which is equal to 
$$
\sum_{{i,j\ge 0, \ i+j >0 , \ i+j+pn_{\beta\alpha}\equiv 0 \ {\rm mod} p, \ i+j \le
p(n_{\beta\alpha}-1)}}1
=\frac{1}{2}[n_{\beta\alpha}]\left(p [n_{\beta\alpha}] +2 - p\right) + p[n_{\beta\alpha}]\{n_{\beta\alpha}\}
-\delta_{p\{n_{\beta\alpha}\},0}.
$$
Summarizing, the contribution of the $L$  terms to the index is
\be\label{l_prime}
 l^\prime_{\alpha ,\beta}=\left\{
\begin{array}{ll}\displaystyle
\frac{1}{2}[n_{\alpha\beta}]\left(p [n_{\alpha\beta}] +2 - p\right) + p[n_{\alpha\beta}]\{n_{\alpha\beta}\}
\quad & \text{if}\quad  n_{\alpha\beta}\ge 0,
\\[10pt]
\displaystyle 
\frac{1}{2}[n_{\beta\alpha}]\left(p [n_{\beta\alpha}] +2 - p\right) + p[n_{\beta\alpha}]\{n_{\beta\alpha}\}
-\delta_{p\{n_{\beta\alpha}\},0}
\quad & \text{otherwise.} \end{array}
\right. 
\ee

The only operation to be yet performed is  the sum over the boxes of $Y_\alpha $ in
(\ref{nondiagonal_tangent_space}). A careful analysis shows that the contribution is $\mid
Y_\alpha  \mid +\mid Y_\beta \mid -n^\prime_{\alpha ,\beta}$, where
\be\label{n_prime}
 n^\prime_{\alpha ,\beta}=\left\{ \begin{array}{l} 
\text{$\sharp$ of columns of $Y_\alpha$ that are longer than $k_\alpha  - k_\beta$ if  $k_\alpha  - k_\beta\ge
0$,} \\
\text{$\sharp$  of columns of  $Y_\beta$ that are longer than  
$k_\beta - k_\alpha -1$ otherwise. }\end{array}
\right.
\ee
With   this information we may compute the desired Poincar\'e polynomial.
\begin{thm} \label{poipol} The Poincar\'e polynomial of $\widetilde\M^p(r,k,n)$ is
$$
P_t(\widetilde\M^p(r,k,n))=\sum_{\text{\rm fixed}\atop\text{\rm points}} \prod_{\alpha =1}^r t^{2(\mid
Y_\alpha \mid -l(Y_\alpha ))} \prod_{i=1}^\infty
\frac{t^{2 (m^{(\alpha)}_i+1)}-1}{t^2-1}
\prod_{\alpha <\beta}t^{2(l^\prime_{\alpha ,\beta}+\mid Y_\alpha 
\mid +\mid Y_\beta \mid -n^\prime_{\alpha ,\beta})}\,.
$$ \label{thm}
\qed \end{thm}
Here  $m_i^{(\alpha)} $ is the number 
of columns in $Y_\alpha$ that have length $i$.

\begin{corol}
The generating function for the Euler characteristics of the moduli spaces $\widetilde\M^p(r,k,n)$
is
$$
\sum_{k,n} P_{-1}(\widetilde\M^p(r,k,n)) \, q^{n+\frac{pk^2}{2r}} z^k = \left(\frac{\theta_3(\frac{v}{p}|\frac{\tau}{p})}
{\hat\eta(\tau)^2}\right)^r
$$
with $q={\rm e}^{2\pi i \tau}$ and  $z=e^{2\pi i v}$.
\end{corol}
\begin{proof}
It just follows by setting $t=-1$ in the above expression for $P_t(\widetilde\M^p(r,k,n))$ and the 
standard formulas for the (quasi)-modular functions
\begin{eqnarray}
\theta_3(v|\tau) &=& \sum_{n\in\mathbb{Z}} q^{\frac{1}{2}n^2} {\rm e}^{2\pi i v n} \nonumber \\
\hat\eta(\tau) &=& \prod_{l=1}^{\infty} (1 - q^l)\,. \nonumber 
\end{eqnarray}
\end{proof}

\bigskip
\section{Some comparisons and consequences} \label{checks}
\subsection{Irreducibility of the moduli space}
First, we note that our computation of the Poincar\'e polynomial of the spaces
$\M^p(r,k,n)$ allows us to conclude that these spaces are irreducible.
\begin{thm} For all $p\ge 1$, and all values of the parametes $r$, $k$, $n$, the moduli space $\M^p(r,k,n)$ of framed sheaves on the Hirzebruch surface $\FF_p$ is irreducible.
\end{thm} 
\begin{proof} By letting $t=0$ in the formula for the Poincar\'e polynomial we see that
 $\M^p(r,k,n)$ is connected, and since it is smooth (Proposition \ref{moduli}), it is also irreducibile.
 \end{proof}

\subsection{Hilbert schemes of points}
For $r=1$ Theorem \ref{poipol} yields the Poincar\'e polynomial of the Hilbert schemes
of points of the spaces $X_p$. The generating function for these is given by
\begin{equation}\label{same}\sum_{n=0}^\infty P_t(X_p^{[n]} )q^n =  \prod_{j=1}^{\infty} 
\frac{1}{(1-t^{2j-2}q^j)(1-t^{2j}q^j)}\,
\end{equation}
This is independent of $p$, according to what was noted  in \cite[Cor.~7.7]{NakaBook}:  the
Betti numbers of the Hilbert scheme of the total space of a line bundle on a Riemann
surface $\Sigma$ do not depend on the line bundle. Indeed formula \eqref{same} is the formula of \cite[Cor.~7.7]{NakaBook} for $\Sigma=\mathbb P^1$.

\subsection{The rank 2 case}
One can  
examine $\M^p(2,k,n)$  in some detail.
 In the case
$p=1$ our result coincides with   Corollary 3.19 of \cite{NY-L},
where it is also shown (identity (3.20)) that  the generating
function of the Poincar\'e polynomial can be further factorized
  to obtain an elegant product formula. Our computations
  also agree with some explicit characterizations of the
  spaces $\M^1(2,k,n)$ given in \cite{BM}.

\subsection{The case $p=2$}
This is examined in
\cite{sasaki}.  Besides usual line bundles $\cO (k_{\alpha} C)$
with integer $k_\alpha$, the author (without much mathematical
justification) considers there also line bundles with half-integer
Chern class $k_\alpha$ (see discussion after eq. (3.35) of
\cite{sasaki}). This is what we get in the ``stacky'' case for $p=2$.

It is interesting that incorporating the  terms
corresponding to the half-integer $k_\alpha$ again leads to a
product formula (see (2.62)-(2.64) of \cite{sasaki}). This
strongly suggests that one should be able to write a general
factorized expression for the the generating function of the
Poincar\'e polynomial  in the case  of the stacky Hirzebruch surfaces for all values of $p$.
 
\subsection{The case of minimal discriminant} 
We would  also like to compare our computation with the result in Proposition
\ref{boundprop}, which implies that whenever $n$ reaches the lower bound
$N=\frac{p k}{2r} (r-k)$, the moduli space $\M^p(r,k,N)$ is homotopic to the Grassmannian
variety $G_k(r)$. Coherently with the fact that all sheaves in $\M^p(r,k,N)$ are locally free,
in computing the Poincar\'e polynomial for this case one only considers
empty Young tableaux, and takes values  $k_\alpha$ that
obey the condition
$$
\sum_{\alpha=1}^rk_\alpha=\sum_{\alpha=1}^rk_\alpha^2.
$$
This  implies that each $k_\alpha$ is
either $0$ or $1$. Thus the fixed points in the case (\ref{bound})
are in a one-to-one correspondence with the sequences
$\vec{k}=(k_0,k_1\cdots ,k_r)$ with $k_\alpha\in \{0,1\}$ and
$\sharp 1=k$. Such sequences may be conveniently  represented as
Young tableaux (below denoted as $Y_{\vec{k}}$) embedded in a
rectangle of size $(r-k)\times k $ as follows: starting from the
left top of the rectangle successively draw  a vertical (horizontal)
line segment of unit size if you encounter $0$ (1). According to
Theorem  \ref{thm}   Poincar\'e polynomial for
this   case is
\begin{equation}
P_t=\sum_{\vec{k}}\prod_{\alpha <\beta}t^{2l'_{\alpha,\beta}}
\label{Poincare1}
\end{equation}
According to eq. (17) all $l'_{\alpha,\beta}=0$, besides the cases
with $k_\alpha =1$ and $k_\beta =0$ for which
$l'_{\alpha,\beta}=1$. But the number of latter cases exactly
matches with the number of hooks of the given Young tableau. Since
the number of hooks is the same as the number of boxes, we obtain
\begin{equation}
P_t=\sum_{\vec{k}}t^{2 |Y_{\vec{k}}|}. \label{Poincare2}
\end{equation}
Thus the Poincare polynomial is nothing but the generating
function (with respect to the parameter $q=t^2$) of the partitions
with number of parts (i.e.length) $\le r-k$ and with largest part
$\le k$. This generating function is known to coincide with
the ``$q$-binomial coefficient" \cite[p.~27]{Mac}
$$ \left[\begin{array}{cc} r \\ k \end{array}\right]
=\frac{(1-q^{r-k+1})\cdots
(1-q^r)}{(1-q)\cdots(1-q^k)},
$$
which shows that the Poincar\'e polynomial (\ref{Poincare2})
coincides with the Poincar\'e polynomial of the Grassmannian
$G_k(r)$.  
\par

\end{document}